\documentclass[12pt]{amsart}
\usepackage{a4wide, amssymb, bbm, calligra, mathrsfs}
\usepackage[2cell,arrow,curve,matrix]{xy}
\UseHalfTwocells \UseTwocells

\let\To\Rightarrow
\let\lra\longrightarrow

\def\A{{\mathbb A}}
\def\B{{\mathbb B}}
\def\C{{\mathbb C}}
\def\Z{{\mathbb Z}}

\def\scg{{\frak P}}

\def\A{{\mathbb A}}
\def\B{{\mathbb B}}
\def\C{{\mathbb C}}

\def\G{{\mathbb G}}

\def\Z{{\mathbb Z}}

\let\al\alpha
\let\bb\beta

\let\Dd\Delta

\let\d\partial

\def \ee{\mathop{\epsilon}\nolimits}

\let\x\times

\let\then\Rightarrow

\let\then\Rightarrow

\def\cok{\mathop{\sf Coker}\nolimits}

\def\sh{\mathop{\sf Shukla}\nolimits}

\def \ker{\mathop{\sf Ker}\nolimits}
\def \ext{\mathop{\sf Ext}\nolimits}

\def\xto#1{\xrightarrow[]{#1}}

\newtheorem{Pro}{Proposition}
\newtheorem{Le}[Pro]{Lemma}
\newtheorem{The}[Pro]{Theorem}
\newtheorem{Co}[Pro]{Corollary}
\theoremstyle{definition}
\newtheorem{De}[Pro]{Definition}
\theoremstyle{remark}

\newtheorem{Rem}[Pro]{Remark}

\def\ab{_{\mathrm{ab}}}

\def\al{{\alpha}}
\def\bb{{\beta}}

\def\st{{\frak S}{\frak T}}
\def\pst{{\frak P}{\frak S}{\frak T}}

\def\u{{\mathfrak{U}}}
\def\vu{{\mathfrak{V}}}

 \def\ta{{\frak T}}

\def\ab{{\mathbbm{Ab}}}

\def\psh{{\mathbbm{Psh}}}
\def\sh{{\mathbbm{Sh}}}

\def\xyma{\xymatrix@M.7em}

\let\x\times

\def \p{\mathop{\mathcal P}\nolimits}
\def \q{\mathop{\mathcal Q}\nolimits}
\def\xymat#1{\begin{aligned}\xymatrixcolsep{3pc}\xymatrix{#1}\end{aligned}}
\def \f{\mathop{\mathcal F}\nolimits}
\def \Im{\mathop{\sf Im}\nolimits}

\def\down{\downarrow{}}

\begin{document}
\title{ Cohomology with coefficients in stacks}
\author[M. Jibladze]{Mamuka Jibladze}
\address{Razmadze Mathematical Institute, Tbilisi, Georgia} \email{jib@rmi.ge}
\author[T. Pirashvili]{Teimuraz  Pirashvili}
\address{
Department of Mathematics\\
University of Leicester\\
University Road\\
Leicester\\
LE1 7RH, UK} \email{tp59-at-le.ac.uk}
\thanks{Research was partially supported by the GNSF Grant
ST08/3-387}

\maketitle

\begin{abstract}
Cohomology of a topological space with coefficients in stacks of
abelian 2-groups is invented. This theory  extends the classical
sheaf cohomology. An application is given to twisted sheaves.

\bigskip


\end{abstract}
\section{Introduction}

The aim of this work is to extend classical theory of sheaf
cohomology \cite{tohoku} to stack cohomology. Recall that  sheaves
usually take values in sets, but in order to define the cohomology
$H^n(X,F)$, $n\in \Z$ for a sheaf $F$ one needs to assume that $F$
has values in the category of abelian groups. Of course there is an
important generalization of the theory when $F$ has values in not
necessarily abelian groups, but then $H^n(X,F)$ is defined  only for
$n=0,1$ and perhaps $n=2$. Quite similarly, a stack $\f$ usually
takes values in the 2-category of groupoids, which we think as
2-sets and in order to define the cohomology $H^n(X,\f)$, $n\in \Z$
one needs to assume that $\f$ takes values in the 2-category of
abelian 2-groups. Again if one restrict stacks with values in not
necessarily commutative 2-groups then one can still define
$H^n(X,\f)$ but only for few values of $n$. The nonabelian theory
will be developed elsewhere and here we restricts ourselves to the
case when $\f$ takes values in the 2-category of abelian 2-groups.

We recall that an abelian 2-group (known also as Picard category or
symmetric categorical group) is a categorification of the notion of
abelian group. These objects were invented by Grothendieck and
Deligne in the sixties \cite{SGA}. The basic result on abelian
2-groups was proved in the thesis  of Sinh \cite{sinh} written under
the advice of Grothendieck, and states that the 2-category of
abelian 2-groups is 2-equivalent to the 2-category of 2-stage
spectra, see also the recent account in \cite[Appendix B]{HS}. These
objects play important role in many aspects of homotopy theory,
arithmetic and geometry, see for example  \cite{baues},
\cite{beil2}, \cite{zhu1}.

In this paper  for any topological space $X$ and for any stack of
2-abelian groups we define abelian groups $ H^n(X,\f), n\in \Z$.  In
case when $\f$ is a sheaf considered as a discrete stack the groups
$H^*(X,\f)$ coincide with the classical sheaf cohomology. Our
cohomology shares lots of properties with sheaf cohomology, including
the long exact sequence associated to a suitably defined extension of stacks and behavior on suitably
defined injective objects. However unlike the classical case
these properties do not characterize the groups $ H^*(X,\f)$ in  a
unique way. To avoid this difficulty we also introduce the abelian
2-groups ${\bf H}^n(X,\f)$, $n\in \Z$. The group $H^n(X,\f)$ can
be seen as the group of connected components of ${\bf H}^n(X,\f)$,
$n\in \Z$. The 2-groups ${\bf H}^*(X,-)$ form a suitably defined
long exact sequence and vanish on injective objects. As in the
classical case, we prove that these two facts characterize stack
cohomology ${\bf H}^n(X,-)$ in a unique way. This important result
is based on existence of enough injective objects in the abelian
2-category of abelian 2-groups -- a surprising fact recently
discovered by the second author \cite{2der}. One of our results
claims that if $\A$ is an abelian 2-group considered as a constant
stack then the groups $H^*(X,\A)$ are homotopy invariant. We deduce
this result by proving that there is an isomorphism
$$H^*(X,\A)\cong H^*(X,{\sf sp}(\A))$$ where on the right hand side
$H^*$ denotes the cohomology of $X$ with coefficients in spectra as
it is defined in homotopy theory \cite{adams} and ${\sf sp}(\A)$ is
a spectrum associated to $\A$ according to \cite[Proposition
B.12]{HS}.

One needs to make the reader aware of the fact that our  abelian 2-groups are not assumed to be strictly
commutative as many authors do. Deligne \cite{SGA} also considered
stacks of abelian 2-groups, but he assumes that abelian 2-groups are
strictly commutative and proves that any stack of strictly
commutative 2-groups is equivalent to a chain complex of sheaves of
length one and hence corresponding cohomology is isomorphic to the
obvious hypercohomology groups. If we drop the strict
commutativity then the classical homological algebra technique is
not enough to define cohomology in this generality and we have to
use the machinery of the two dimensional homological algebra,
recently developed in \cite{mamuka_adams}, \cite{2-ch}, \cite{2der}.
In that latter paper the classical theory of the derived functors and
$Ext$'s \cite{CE} was extended to the framework of abelian 2-categories
\cite{2der}, which are certain 2-categories having properties
similar to the 2-category of abelian 2-groups. One of the our main
results claims that the stack cohomology can be described using the
secondary $Ext$ as they defined in \cite{2der}. Of course this is a
2-categorical version of Grothendieck's result in \cite{tohoku}.

The paper is organized as follows. After some preliminaries we
introduce the \v{C}ech type cohomology with coefficients in
prestacks of abelian 2-groups. We also modify  Berishvili's approach
to (pre)sheaf cohomology \cite{beri} for prestacks and we prove that
for paracompact spaces the two approaches give equivalent theories.
In the next section we prove that these cohomologies in fact depend
only on associated stacks. In Section 5 we relate these objects with
Ext in the 2-categorical sense. In the last section we give an
application to twisted sheaves and discriminants.

In this paper we restricted ourselves to topological spaces, but
this restriction is not really necessary and one can develop
cohomology with coefficients in stacks of abelian 2-groups for any
Grothendieck site, based on hypercovers instead of Berishvili
covers. In fact in our forthcoming publication we will introduce
cohomology of 2-toposes, which will generalize not only the theory
developed in this paper but also topos cohomology.

\section{Preliminaries on 2-dimensional algebra}

\subsection{Preliminaries on abelian  2-groups and abelian 2-categories} An abelian  2-group is
a groupoid $\A$ equipped with a symmetric monoidal structure
$+:\A\x\A\to \A$ such that for any object $x$ the endofunctor
$x+(-):\A\to \A$ is an equivalence of categories. The symmetry
constraints are denoted by $c_{x,y}:x+y\to y+x$. An abelian 2-group
is called \emph{strictly commutative} provided $c_{x,x}=id_x$.

It would be convenient to think of abelian  2-groups as
2-dimensional analogues of abelian groups. For any abelian  2-group
$\A$ the set of components $\pi^0(\A)$ (written also $\pi_0(\A)$) of
$\A$ has a natural abelian group structure, while the automorphism
group $\pi^{-1}(\A)$ (written also $\pi_1(\A)$) of the zero object
of $\A$ is commutative. Abelian 2-groups form a groupoid enriched
category $\scg$, where 1-morphisms (called simply morphisms) are
symmetric monoidal functors and 2-morphisms (called tracks) are
monoidal transformations. As in any groupoid enriched category, a
morphism $f:\A\to \B$ is an equivalence if and only if there is a
morphism $g:\B\to \A$ and tracks $1_A\then gf$ and $1_B\then fg$.
One easily sees that a morphism $\A\to \A_1$ of abelian 2-groups is
an equivalence of abelian 2-groups if and only if $\A\to \A_1$
yields an isomorphism of abelian groups $\pi^{i}\A\to \pi^i\A_1$ for
$i=0,-1$.

Any abelian group considered as a discrete category is an abelian
2-group. More generally, for any homomorphism of abelian groups
$f:A^{-1}\to A^0$, we have an abelian 2-group ${\sf K(f)}$. Objects
of the category ${\sf K(f)}$ are just elements of $A^0$. If $a,b\in
A^0$, then a morphism from $a$ to $b$ is an element $x\in A^{-1}$
such that $f(x)=b-a$. The composition and the monoidal structure in
${\sf K(f)}$ are induced from the addition in $A^i$, $i=0,1$. It is
clear that $\pi^0{\sf K(f)})={\sf Coker}(f)$ and $\pi^{-1}{\sf
K(f)})={\sf Ker}(f)$. Observe that ${\sf K(f)}$ is a strictly
commutative 2-group. It is well-known that  any strictly commutative
2-group is  equivalent to an abelian 2-group $K(f)$ for some $f$
\cite{SGA}.

The role of the additive group of integers in the 2-dimensional
algebra is plaid by an abelian 2-group $\Phi$. Objects of the
groupoid $\Phi$ are integers; if $n$ and $m$ are integers then there
are no morphisms between them if $m\not = n$, while the automorphism
group of the object $n$ is the cyclic group of order two
$\{+1,-1\}$. The monoidal structure is induced by the group
structure of integers. The associativity and unitality constraints
are the identity morphisms while the commutativity constraint
$n+m\to m+n$ is $(-1)^{mn}$. As we see from the definition $\Phi$ is
not strictly commutative.

One easily observes that for any abelian 2-groups $\A$ and $\B$ the
hom-groupoid $\scg(\A,\B)$ has a canonical structure of an abelian
2-group \cite{SGA}. The 2-category $\scg$ possesses kernels in the
2-dimensional sense. The following construction goes back to Gabriel
and Zisman \cite{gz}. Let $f:\A\to \B$ be a morphism of 2-groups.
Objects of the groupoid $\ker(f)$ are pairs $(a,\al)$, where $a$ is
an object of $\A$ and $\al:0\to f(a)$ is a morphism in $\B$. A
morphism $(a,\al)\to (b,\bb)$ in $\ker(f)$ is a morphism
$\gamma:a\to b$ in $\A$ such that $f(\gamma) \al=\bb$. The
compositions of morphisms as well as the monoidal structure  in
$\ker(f)$  are induced from $\A$. Observe that we have a canonical
functor $k_f:\ker(f)\to \A$ and a canonical track $\kappa_f:0\then
fk_f$, defined by
$$k_f(a,\al)=a, \ \ \kappa_f(a,\al)=\al.$$
The following important exact sequence was first constructed by
Gabriel and Zisman (see p.84 in \cite{gz})
$$
0\to
\pi^{-1}(\ker(f))\to\pi^{-1}(\A)\to \pi^{-1}(\B)\to
\pi^0(\ker(f))\to\pi^0(\A)\to \pi^0(\B).
$$
It is functorial in $f$ in the following sense. Let
$$\xymat{\B\ar[r]^{g}\drtwocell\omit{^\ee} &\B'
\\
\A\ar[u]^f\ar[r]_{t} &\A'\ar[u]^{f'}}
$$
be a diagram in the 2-category $\scg$. Thus $\ee:gf\then f't$ is a
track. Then the assignment $(a,\al)\mapsto (t(a),\al')$ defines a
morphism of abelian 2-groups $\ker(f)\to\ker(f')$. Here $\al'$ is
the following composite
$$\xymatrix{0\ar@{=>}[r] &g(0) \ar@{=>}[r]^{g(\al)}&gf(a) \ar@{=>}[r]^{\ee_a}& f't(a)}.$$

Based on  the construction of kernels of abelian 2-groups, one can
introduce the notion of the kernel in any 2-category $\ta$ enriched
in $\scg$ \cite{dupont}, \cite{nakaoka}. Let $f:A\to B$ be a
morphism in $\ta$. A diagram
$$
\xymatrix{ K\ar[r]_{k}\rruppertwocell<12>^{0}{^\kappa}
&A\ar[r]_{f}&B}
$$
is a kernel of $f$ if for any object $X\in \ta$ the induced functor
$$\xi:\ta(X,K)\to \ker(f^X)$$
is an equivalence of abelian 2-groups. Here $f^X:\ta(X,A)\to
\ta(X,B)$ is the induced morphism of abelian 2-groups and
$\xi(g:X\to K)=(kg,g^*(\kappa))$.  Of course these notions are
compatible, meaning that for $f:\A\to \B$ in $\scg$, the triple $(\ker(f),k_f,\kappa_f)$ is the
kernel of $f$ in $\scg$ in this sense. To simplify notation, we will say that
$K$ is the kernel of $f$. By duality one introduces cokernels.
According to \cite{vitale} the 2-category $\scg$ possesses also
cokernels and is a prototype of abelian 2-categories \cite{dupont},
\cite{ab-2-ab}.

A morphism $f:A\to B$ in a 2-category $\ta$ is called \emph{faithful} (resp.
\emph{cofaithful}) provided the induced functor $f^X:\ta(X,A)\to
\ta(X,B)$ (resp. $f_X:\ta(B,X)\to \ta(A,X)$) is faithful. For
an abelian 2-category $\ta$, a morphism $f:A\to B$ in $\ta$ is faithful
(resp. cofaithful) iff the induced homomorphism of abelian groups $\pi^{-1}(\ta(X,A))\to \pi^{-1}(\ta(X,B))$
(resp. $\pi^{-1}(\ta(B,X))\to \pi^{-1}(\ta(A,X))$) is a monomorphism.

A morphism $f:\A\to \B$ of abelian 2-groups is faithful (resp. cofaithful) iff the induced morphism
$\pi^{-1}\A\to \pi^{-1}\B$ is a monomorphism (resp. $\pi^{0}\A\to\pi^0\B$ is an epimorphism).

An object $I$ of an abelian 2-category $\ta$ is called
\emph{injective} provided for any faithful morphism $f:A\to B$ and a
morphism $g:A\to I$ there exists a morphism $h:B\to I$ and a track
$hf\then g$. Dually, an object $P$ of an abelian 2-category $\ta$ is
called \emph{projective} provided for any cofaithful morphism
$f:A\to B$ and a morphism $g:P\to B$ there exists a morphism $h:P\to
A$ and a track $fh\then g$. We will say that $\ta$ has \emph{enough
injective} (resp. \emph{enough projective}) objects if for any object
$A$ there exists an injective (resp. projective) object $I$ (resp.
$P$) and a faithful (resp. cofaithful) morphism $A\to I$ (resp.
$P\to A$).

It was proved in \cite{2der} that the 2-category $\scg$ possesses
enough injective and projective objects. For example the abelian
2-group $\Phi$  is the unique (up to equivalence) small,
indecomposable, projective generator of $\scg$ \cite{2der}.

An \emph{extension} of an object $A$ by an object $C$ in an abelian
2-category $\ta$ is  a  triple $(i,p,\al)$ where $i:A\to B$ and
$p:B\to C$ are morphisms in $\ta$ and $\al:0\then pi$ is a track,
such that $p$ is cofaithful and $(A,i,\al)$ is equivalent to the
kernel of $p$. If this is the case, then $i$ is faithful and $C$ is
equivalent to the cokernel of $i$ \cite{2-ch}, \cite{dupont},
\cite{nakaoka}. We sometimes depict such a situation by
$$0\to A\to B\to C\to 0$$
without indicating the track $\al$.

We also need a more general notion, which is called 2-exactness
\cite{2-ch}. For simplicity we consider only the case when
$\ta=\scg$. Assume we have a diagram of abelian 2-groups and tracks
$$
\xymatrix{ \A\ar[r]_f\rruppertwocell<12>^{0}{^\al} &\B\ar[r]_{g}&\C}.
$$
We will say that it is 2-exact at $\B$ provided the induced morphism
$\A\to \ker(g)$ is full and essentially surjective, in other words
it yields an  isomorphism on $\pi^0$ and epimorphism on $\pi^{-1}$.
It follows then that the sequence of abelian groups
$$\pi^i(\A)\to \pi^i(\B)\to \pi^i(\C)$$
is exact at $\pi^i(\B)$ for $i=0,-1$ \cite{2-ch}.

We recall the description of colimits of pseudofunctors following \cite{br}, pp. 192-193. For simplicity we will
only consider pseudofunctors over directed partial ordered sets considered as categories, as this will be the only
case we need. Let $I$ be a directed category and let $C_{()}$ be a pseudofunctor from $I$ to the 2-category of groupoids.
It is thus given by the family of groupoids $(C_i)_{i\in I}$ indexed by objects of $I$, functors $\alpha_*:C_i\to C_j$ for $\alpha:i\to j$ in $I$,
and natural isomorphisms $\Theta_{\beta,\alpha}:\beta_*\alpha_*\then(\beta\alpha)_*$ for $\alpha:i\to j$, $\beta:j\to k$, satisfying appropriate coherence conditions.
The set of objects of the colimit category $C={\sf colim}_{i\in I}C_i$ is just the disjoint
union of the sets of objects of $C_i$, $i\in I$. To describe morphisms we need some notations.
For objects $i_1,i_2$ of the category $I$ we let $i_1\down I\down i_2$ be the category with objects pairs of arrows
$i_1\to j\leftarrow i_2$ and obvious morphisms between these.
For objects $P_1\in C_{i_1}$ and $P_2\in C_{i_2}$, the set of
morphisms from $P_1$ to $P_2$ in $C$ is just the colimit of the functor
$i_1\down I\down i_2\to {\sf Sets}$ which assigns to the object
$(\alpha_1:i_1\to j, \alpha_2:i_2\to j)$ of $i_1\down I\down i_2$ the set
$Hom_{C_j}({\alpha_1}_*(P_1),{\alpha_2}_*(P_2))$, and to a morphism $(\alpha_1,\alpha_2)\to(\gamma\alpha_1,\gamma\alpha_2)$ the map given by
$$
\left({\alpha_1}_*(P_1)\xto f{\alpha_2}_*(P_2)\right)\mapsto
\left({(\gamma\alpha_1)}_*(P_1)\xto{\Theta_{\gamma,\alpha_1}^{-1}}\gamma_*{\alpha_1}_*(P_1)\xto{\gamma_*(f)}\gamma_*{\alpha_2}_*(P_2)\xto{\Theta_{\gamma,\alpha_2}}{(\gamma\alpha_2)}_*(P_2)\right).
$$

Now assume that $C_{()}$ takes values in $\scg$, i.~e. is an object of the 2-category $\scg^I$ of pseudofunctors from $I$ to $\scg$, so that each $C_i$ is an abelian 2-group, and the data ($\alpha_*$, $\Theta_{\beta,\alpha}$) are compatible with the monoidal structures on the $C_i$. For each pair of objects
$(i,j)$ of $I$ we choose an object $\xi(i,j)$ and morphisms
$\alpha_{i,j}:i\to \xi(i,j)$, $\beta_{i,j}:j\to \xi(i,j)$ in $I$. We also
choose an object $i_0$ of $I$. Having such choices made we define the
bifunctor
$$+:C\x C\to C$$
as follows: on objects it is given by $$(P_1,P_2)\mapsto
(\alpha_{i,j})_*(P_1)+(\beta_{i,j})_*(P_2),$$ where $P_1$ and $P_2$ are
objects of $C_i$ an $C_j$ respectively. This assignment has obvious
extension to morphisms. The bifunctor $+$ together with the object
$0_{i_0}$ is part of an abelian 2-group structure on $C$.
For example, the associativity constraints are obtained by choosing,
using directedness of $I$, objects $\xi(i,j,k)$ and morphisms $\alpha_{i,j,k}:\xi(\xi(i,j),k)\to\xi(i,j,k)$, $\beta_{i,j,k}:\xi(i,\xi(j,k))\to\xi(i,j,k)$ for each triple $(i,j,k)$ of objects of $I$. This then yields, for objects $P_1$, $P_2$, $P_3$ of, respectively, $C_i$, $C_j$ and $C_k$, isomorphisms $(P_1+P_2)+P_3\to P_1+(P_2+P_3)$ as elements of the colimit over $\xi(\xi(i,j),k)\down I\down\xi(i,\xi(j,k))$, using the isomorphisms $\Theta$ and the associativity constraints in $C_{\xi(i,j,k)}$.

Observe that different choices of $\xi$, $\alpha$, $\beta$ give rise to an equivalent abelian
2-group. In this way one obtains a pseudofunctor
$${\sf colim}:\scg^I\to\scg.$$ One observes that
$$\pi^n({\sf colim}_{i\in I} C_i)={\sf colim}_{i\in I}\pi^n(C_i), \ n=0,-1.$$
It follows that ${\sf colim}:\scg^I\to \scg$ is exact and for any
cofinal subcategory $J$ of $I$ the obvious morphism ${\sf colim}_{j\in
J}C_j\to {\sf colim}_{i\in I}C_i$ is an equivalence of abelian
2-groups.

\subsection{2-chain complexes} What is important for us is that there is
a way of doing  homological algebra in $\scg$, or more generally in
any abelian 2-category $\ta$.  More precisely, a \emph{2-cochain
complex} $(\A_*,d,\partial)$ in $\scg$ is a diagram of the form
$$
\xymatrix{ \cdots\ar[r]\rrlowertwocell<-12>_{0}{}
&\A^{n-1}\ar[r]|-{d^{n-1}}\rruppertwocell<12>^{0}{^\partial^n}
&\A^{n}\ar[r]|-{d^n}\rrlowertwocell<-12>_{0}{_{}{\hskip1.2em\d^{n+1}}}
&\A^{n+1}\ar[r]|-{d^{n+1}}\rruppertwocell<12>^{0}{^{}}
&A^{n+2}\ar[r] &\cdots }
$$
i.~e., a sequence of abelian  2-groups $\A^n$, maps $d^n:\A^{n}\to
\A^{n+1}$ and tracks $\partial^n:d^{n+1}d^{n}\then0$, $n\in\Z$, such
that for each $n$ the tracks
$$
\xymatrix{d^{n+1}d^nd^{n-1}\ar@{=>}[r]^{\ \ \ \ \ d^{n+1}\d^n}&d^{n+1}0\ar@{=>}[r]^{\ \ \ \equiv}&0}
$$
and
$$
\xymatrix{d^{n+1}d^nd^{n-1}\ar@{=>}[r]^{\ \ \ \ \ \partial^{n+1}d^{n-1}}&0d^{n-1}\ar@{=>}[r]^{\ \ \ \equiv}&0}
$$
coincide.

For any 2-cochain complex $(\A^*,d,\partial)$ and any integer $n$,
there is a well-defined abelian  2-group called $n$-th cohomology
${\bf H}^n(\A^*)$ of $\A^*$ (see \cite{2-ch}). We recall here the
definition of these abelian 2-groups. Assume we have morphisms of
abelian 2-groups $\xymatrix{\A\ar[r]^{f}&\B\ar[r]^{g} &\C}$ and a
track $\al:0\then gf$. Then we have the diagram
$$\xymat{\B\ar[r]^{g}\drtwocell\omit{^\al} &\C
\\
\A\ar[u]^f\ar[r] &0\ar[u]}
$$
which yields a morphism of abelian 2-groups $\al':\ker(f)\to \Omega
\C$, where $\Omega \C=\ker(0\to \C)$. We let $\ker(f,\al)$ be the
kernel of $\al'$ and call it the \emph{relative kernel}. This
construction makes sense in any abelian 2-category as well. In
particular one can talk about relative cokernels.  For a 2-cochain
complex $(\A^*,d,\partial)$ we first take the relative kernel
$\ker(d^n,\partial^{n+1})$. It comes with a natural morphism
$d':\A^{n-1}\to \ker(d^n,\partial^{n+1})$ and a track
$\partial':0\to d'\circ d^{n-2}$ and ${\bf H}^n(\A^*)$ is defined to
be $\cok(d',\partial')$ \cite{2-ch}. Following
\cite{mamuka_adams} we call ${\bf H}^*(\A^*)$ the \emph{secondary
cohomology} of $\A^*$.

We also put $$H^n_U(\A^*):=\pi^0({\bf H}^n(\A^*))$$ These groups are
known as \emph{Takeuchi-Ulbrich cohomology} \cite{2-ch} and first
were defined in \cite{ulb},\cite{tak_ulb}. We have an isomorphism
$$\pi^{-1}({\bf H}^{n+1}(\A^*))\cong H^n_U(\A^*)$$
and an exact sequence (called TU-exact sequence) of abelian groups
$$\cdots \to H^{n+1}(\pi^{-1}(\A^*))\to H^n_U(\A^*)\to H^n(\pi^0(\A^*))\to
H^{n+2}(\pi^{-1}(\A^*))\to \cdots $$
\begin{Le}\label{TUzero} If $H_U^*(\A^*)=0$, then ${\bf H}^*(\A^*)$ is equivalent
to the zero object. More generally, if $f:\A^*\to \B^*$ is a morphism of
2-chain complexes such that the induced morphism $H_U^*(\A^*)\to
H_U^*(\B^*)$ is an isomorphism of abelian groups then the morphism of
abelian 2-groups ${\bf H}^*(\A^*) \to {\bf H}^*(\B^*)$ is an
equivalence.
\end{Le}

\begin{proof} A morphism of  abelian 2-groups $\C\to \C'$ is an equivalence  if and
only if the induced morphism $\pi^i(\C)\to \pi^{i}(\C')$ is an
isomorphism for $i=0,-1$. Thus the result follows from the
isomorphism $\pi^{-1}({\bf H}^{n+1}(\A^*))\cong H^n_U(\A^*)$.
\end{proof}
Of course there is also a notion of a morphism of 2-cochain
complexes as well as a homotopy between two parallel morphisms of
2-cochain complexes, with expected properties. One has also an
analogue of the long cohomological exact sequence in the 2-dimensional
world \cite{2-ch}. In fact the following is one of the main results
of \cite{2-ch}. Assume we have an extension of 2-cochain complexes
$$\xymatrix{0\ar[r] &\A^*\ar[r]^{i^*}& \B^*\ar[r]^{p^*}& \C^*\ar[r]& 0}$$
Here we assume that for each $n$ there are given tracks
$\al^n:0\then p^ni^n$ such that  $\al^n$ and $\al^{n+1}$ are
compatible in the obvious sense. Then there are morphisms ${\bf
H}^n(\C^*)\to {\bf H}^{n+1}(\A^*)$, $n\in \Z$, and appropriate tracks
such that the diagram
$$\cdots \to {\bf H}^n(\A^*) \to {\bf H}^n(\B^*)\to {\bf H}^n(\C^*)\to {\bf
H}^{n+1}(\A^*)\to \cdots$$ is part of a 2-exact sequence.

\subsection{Derived 2-functors}\label{Sec_ext} Let $\ta$ be an abelian 2-category
with enough injective objects and $A$ be an object in $\ta$.
Following \cite{2der}, an \emph{injective resolution} of $A$ is a
morphism $ A\to A^*$ of 2-cochain complexes, which induces
isomorphism on secondary cohomology, where $A$ is considered as a
2-chain complex concentrated in dimension 0 with trivial
differentials $d=0,
\partial=0$ and $(A^*,d,\partial)$ is a 2-cochain complex with injective $A^n$, $n\geq 0$
and  $A^n=0$, $n<0$. Moreover $\partial^n$ is equal to the identity
track for $n<0$. As in the classical case, any object admits an
injective resolution, which is unique up to homotopy. For any
additive pseudofunctor $T:\ta \to \scg$ one obtains  well-defined
additive pseudofunctors  ${\sf R}^n(T):\ta\to \scg$, $n\in \Z$
(called the \emph{secondary right derived 2-functors}) by
$${\bf R}^n(T) (A):={\bf H}^n(T(I^*))$$
where $I^*$ is an injective resolution of $A$. If one takes the
Takeuchi-Ulbrich homology instead, one gets the Takeuchi-Ulbrich right
derived functors, which are denote by $R^nT$, $n\in \Z$. Then for
any extension
$$
0\to  A\xto{i} B\xto{p} C\to  0, \ \ \al:0\then pi
$$
the sequence
$$\cdots \to {\bf R}^{n-1}T(C)\to {\bf R}^{n}T(A)\to {\bf R}^{n}T(B)\to
{\bf R}^{n}T(C)\to \cdots$$ is part of a 2-exact sequence of
abelian 2-groups. Furthermore we have the following exact sequence
of abelian groups
$$\cdots \to {R}^{n-1}T(C)\to {R}^{n}T(A)\to {R}^{n}T(B)\to
R^nT(C)\to \cdots$$ Moreover ${\bf R}^nT=0$ if $n<-1$ and $R^nT=0$
if $n<0$.

\begin{Pro} \label{ax}\cite{2der} Assume  ${\bf T}^n:\ta\to \scg$, $n\in \Z$  is a system of
additive pseudofunctors such that ${\bf T}^n=0$, if $n<-1$. Assume
the following conditions hold

i) for any extension $$
0\to  A\xto{i} B\xto{p} C\to  0, \ \ \al:0\then pi
$$
the sequence
$$\cdots \to {\bf T}^{n}(A)\to {\bf T}^{n}(B)\to {\bf T}^{n}(C)\to
{\bf T}^{n+1}(A)\to \cdots$$ is part of a 2-exact sequence of
abelian 2-groups,

ii) for any injective $I$ one has ${\bf T}^n(I)=0$ for $n>1$ and
$\pi^{0} {\bf T}^1(I)=0$.

Then there exists a natural equivalence of 2-functors
$${\bf R}^n{\bf T}^0\cong {\bf T}^n, \ n\in \Z.$$
\end{Pro}

In particular one can take the 2-functor  ${\bf Hom}(-,B)$ and get
the secondary derived 2-functors ${\bf Ext}_\ta^n(-,B)$ as well as
the Takeuchi-Ulbrich derived functors ${Ext}_\ta^n(-,B)$. For $n=1$
these objects are related to extensions; for a precise result we
refer the reader to \cite{2der}.

\section{Topological significance}

In classical algebraic topology chain complexes usually
arise from (pre)cosimplicial abelian groups, by taking the boundary
operator to be the alternating sum of coface operators. According to
\cite{tak_ulb} a similar construction works in dimension 2 as well.
More precisely, a \emph{precosimplicial object} in the 2-category
$\scg$ is a sequence of abelian 2-groups $\A^n$, $n\geq 0$,
morphisms of abelian 2-groups $d_i:\A^n\to \A^{n+1}$, $0\leq i\leq
n$ and tracks
$$\alpha_{i,j}:d_i\circ d_j\To d_{j+1}\circ d_i, \ \ i\leq j$$
such that for any $i\leq j\leq k$ the following diagram commutes
$$\xymatrix{&d_{j+1}d_id_k\ar@{=>}[dr]^{d_{j+1}(\al)}\\
d_id_jd_k\ar@{=>}[ur]^\al \ar@{=>}[d]_{d_i(\al)}&&d_{j+1}d_{k+1}d_i\ar@{=>}[d]_\al\\
d_id_{k+1}d_j\ar@{=>}[dr]_{\al}&&d_{k+2}d_{j+1}d_i\\
&d_{k+2}d_id_j\ar@{=>}[ur]_{d_{k+2}(\al)}}$$
According to Section 3 of \cite{tak_ulb} one can associate a 2-cochain
complex $C^*(\A^*)$ to a precosimplicial object $\A^*$ of $\scg$ .
More precisely, we have $C^n(\A^*)=\A^n$ and
$d=\sum_{i=0}^{n+1}(-)^i\d^i$ with appropriate $\delta:d^2\then 0$.
By abuse of notation we use the notations ${\bf H}^n(\A^*)$ and
$H^n_U(\A^*)$ instead of ${\bf H}^n(C^*(\A^*))$ and
$H^n_U(C^*(\A^*))$.

It is clear that any pseudofunctor from the category $\Dd$ to the
2-category $\scg$ gives rise to a precosimplicial object.

Recall that \cite{adams} a spectrum $E$ is  a
sequence of topological spaces, or better simplicial sets,  $E_n$
and continuous maps $\Sigma E_n\to E_{n+1}$, $n\in \Z$. A spectrum
$E$ is an $\Omega$-spectrum provided the induced map $E_{n}\to
\Omega E_{n+1}$ is a weak equivalence. Any spectrum $E$ gives rise
to the (generalized) cohomology theory on topological spaces by
$X\mapsto H^n(X,E)$, where $X$ is a topological space. In case when
$E$ is an $\Omega$-spectrum one has
$$H^n(X,E)=[X,E_n]$$

Let $k>0$ be an integer. A spectrum $E$ is called $k$-stage  if
$\pi_i(E)=0$ for $i<0$ and $i\geq k$. It is well known that if $E$
is a 1-stage spectrum corresponding to the abelian group $A$ (that
is $\pi_0(E)=A$) then $H^*(X,E)$ coincides with the classical
singular cohomology $H^*(X,A)$, which by definition is the
cohomology of the cochain complex associated to the cosimplicial
abelian group $A^{Sing_*(X)}$, where $Sing_*(X)$ is the singular
simplicial set of $X$. As we have seen for general $E$  even the
definition of $H^*(X,E)$ uses heavy machinery of homotopy theory.
Unlike to the singular cohomology of $X$ with coefficients in an
abelian group $A$ it is impossible  to get these groups from the
classical homological algebra means (see \cite{ch_th}).

The 2-dimensional algebra gives a  similar result for 2-stage
spectra. To state the corresponding result let us recall that by the
result of \cite{sinh}  the 2-category $\scg$ is 2-equivalent to the
2-category of two-stage spectra (see also Proposition B.12 in
\cite{HS}). If $\A$ is an abelian 2-group, we let ${\sf sp}(\A)$ be
the corresponding spectrum. Below is a hint how to construct  ${\sf
sp}(\A)$.

It is well known that any abelian 2-group is equivalent to one for
which the associativity and unitality constraints are identities.
Let us call such abelian 2-groups strictly associative. Let $n\geq
2$ be an integer. Then the category of strictly associative abelian
2-groups and strict morphisms is equivalent to the full subcategory
of the category of simplicial groups consisting of simplicial groups
$G_*$ whose Moore normalization is nontrivial only in dimensions $n$
and $n+1$ \cite{conduche}. For a strict abelian group $\A$ we let
$T(\A,n)$ be the corresponding simplicial group. A more direct
construction of $T(\A,n)$ can be found in \cite{bcc}. So if one
takes ${\sf sp}(\A)_n$ to be the classifying space of $T(\A,n)$,
$n\geq 2$,  one
 obtains the desired $\Omega$-spectrum. Hence we have
$$\pi_n({\sf sp}(\A))=\begin{cases}\pi^{-n}(\A) & n=0,1\\ 0, &n\not
=0,1.\end{cases}$$ Moreover, by dimension reasons  ${\sf sp}(\A)$
has only one nontrivial Postnikov invariant which is the
homomorphism $\pi^0(\A)/2\pi^0(\A)\to \pi^1(\A)$ induced by the
symmetry constraint $$a\mapsto c_{a,a}\in Aut(a+a)\cong
Aut(0)=\pi_1(\A)$$ where the canonical isomorphism $Aut(0)\to
Aut(b)$ is induced by the functor $b+(-):\A\to \A$.

 The most important spectrum is the sphere
spectrum $S$, and the abelian group $\Z$ can be seen as the zeroth
Postnikov truncation of $S$. If we take the next stage we obtain the
spectrum $S_{\leq 1}$ with properties
$$\pi_0(S_{\leq1})=\Z, \ \ \pi_1(S_{\leq 1})=\Z/2\Z, \ \ \pi_i(S_{\leq 1})=0, \
i\ne0,1$$ From the above description of the Postnikov invariant it
is clear that ${\sf sp}(\Phi)$ and $S_{\leq 1}$ are homotopy
equivalent spectra. Observe also that if $\A$ is a strictly
commutative abelian 2-group, then the Postnikov invariant of ${\sf
sp}(\A)$ is zero. Hence it splits as a product of two spectra ${\sf
sp}(\pi^0(\A))$ and ${\sf sp}(\pi^{-1}(\A))[1]$ and hence we have
$$H^*(X,{\sf
sp}(\A)\cong H^*(X, \pi^0(A))\oplus H^{*+1}(X, \pi^{-1}(A))$$ For
general $\A$ we have the following result.

\begin{The} Let $X$ be a topological space and $\A$ be an abelian 2-group. Then  one has the
natural isomorphisms of abelian groups:
$$H^*(X,{\sf sp}(\A))\cong H^*_U(\A^{Sing_*(X)})$$
\end{The}
\begin{proof}
Without loss of generality one can assume that $\A$ is strictly
associative. Then the result follows from Theorem 4.3  of
\cite{bcc}.
\end{proof}
\section{Preliminaries on presheaves and sheaves of abelian groups}
\subsection{Main definitions} Let $X$ be  a topological space. We let $OP(X)$ be the
category corresponding to the partially ordered set of open subsets
of $X$. A \emph{presheaf} $P$ on $X$ is simply a contravariant
functor from $OP(X)$ to the category of abelian groups $\ab$. The
category of presheaves on $X$ is denoted by $\psh(X)$. The category
$\psh(X)$ is an abelian category with enough projective and
injective objects \cite{tohoku}. A sequence
$$0\to P_1\to P\to P_2\to 0$$
is exact in $\psh(X)$ if and only if for any open subset $U$ of $X$
the sequence
$$0\to P_1(U)\to P(U)\to P_2(U)\to 0$$
is an exact sequence of abelian groups.

 Let $U$ be an open subset of $X$ and let
$\u= \{U_i\}_{i\in I}$ be an open cover of $U$. Recall that the
nerve $N\u$ of this cover  is the simplicial space given by
$$[n]\mapsto \bigsqcup_{i_0,\cdots,i_n\in I} U_{i_0\cdots i_n}$$
with the obvious face and
degeneracy maps induced by the inclusion of open sets.  Here $
U_{i_0\cdots i_n}=\bigcap_{j=0}^nU_{i_j}$. If $P$ is a presheaf on $X$
we can apply $P$ on $N\u$ to obtain a cosimplicial abelian group
$[n]\mapsto \prod_{i_0, \cdots,i_n\in I} P(U_{i_0\cdots i_n})$. The
cohomology of the associated cochain complex is denoted by $H^*(\u,
P)$. The obvious augmentation $N\u\to U$ yields the homomorphisms
$$P(U)\to H^0(\u,P ).$$
A presheaf $P$ is called a \emph{sheaf} provided for any open set
$U$ and any open cover $\u= \{U_i\}_{i\in I}$ of $U$ the canonical
homomorphism of abelian groups $P(U)\to H^0(\u, P)$ is an
isomorphism \cite{tohoku}. We let $\sh(X)$ be the full subcategory
of $\psh(X)$ consisting of sheaves on $X$. It is well known that the
inclusion $\sh(X)\hookrightarrow \psh(X)$ has a left adjoint
$P\mapsto P^+$ which preserves kernels. It follows that $\sh(X)$ is
an abelian category with enough injective objects \cite{tohoku}. A
sequence
$$0\to P\to Q\to R\to 0$$
is exact in $\sh(X)$ if and only if for any  $x\in X$  the sequence
$$0\to P_x\to Q_x\to R_x\to 0$$
is an exact sequence of abelian groups. Here for any presheaf $F$
and point $x\in X$ the group $F_x$ is defined by
$$F_x={\sf colim}_{x\in U}F(U).$$

Any abelian group $A$ gives rise to the constant presheaf with value
$A$, which we denote by $A_c$. The associated sheaf $A_c^+$ is
called \emph{constant sheaf} and by abuse of notation will be
denoted by $A$. According to \cite{tohoku} the group
$\ext_{\sh(X)}^*(\Z,F)$ is called the \emph{cohomology}
 of $X$ with
coefficient in a sheaf $F\in \sh(X)$.

If $F$ is a constant sheaf associated to an abelian group $A$ and
$X$ is a polyhedron then these groups are isomorphic to the singular
cohomology of $X$ with coefficients in $A$ \cite{tohoku}.

Since injective objects in $\sh(X)$ are quite mysterious it is
helpful to use \v{C}ech cohomology. Let $\u$ and $\vu=\{V_j\}_{j\in
J}$ be two covers of $X$. One says that $\vu$ is a \emph{refinement}
of $\u$ (notation $\vu<\u$) if there is a map $\al:J\to I$ such that
$V_j\subset U_{\al(j)}$ for all $j\in J$. Utilizing this map $\al$
we define a morphism of simplicial spaces $$N\al: N\vu\to N\u$$ by
mapping $V_{j_0\cdots j_n}$ to $U_{\al(j_0)\cdots \al(j_n)}$. If
$\bb:J\to I$ is another map with $V_j\subset U_{\bb(j)}$ for each
$j\in J$, then two morphisms $N\al$ and $N\bb$ of simplicial spaces
are homotopic. In fact the homotopy operators $h^k:N_n\vu\to
N_{n+1}\u$, $k=0,\cdots, n$ are defined by mapping $V_{j_0\cdots
j_n}$ to $U_{\al(j_0),\cdots,\al(j_k),\bb(j_k)\cdots \bb(j_n)}$.
Hence for any presheaf $P$ there are canonical homomorphisms
$$H^*(\u, P)\to H^*(\vu, P)$$ and one can define the \v{C}ech
cohomology $\check{H}^i(X,F)$ of $X$ with coefficients in a presheaf
$P$ by
$$\check{H}^i(X,F):={\sf colim}_\u H^*(\u,P)$$
where colimit is taken over all open covers. According to
\cite{tohoku}  for paracompact $X$ and any sheaf $F$ one has an
isomorphism
$$\ext_{\sh(X)}^*(\Z,F)\cong \check{H}^i(X,F).$$
\subsection{Berishvili approach to sheaf cohomology} For arbitrary
$X$ we still have a similar result but we have to use so-called
Berishvili covers  instead \cite{beri}. A \emph{Berishvili cover} of
a topological space $X$ is a function $\al$ which assigns to some
ordered tuples $(x_0,\cdots ,x_n)$ of points of $X$ an open subset
$\al(x_0,\cdots,x_n)$ of $X$ such that the following conditions
i)-iv) hold.

 i) $x_n\in
\al(x_0,\cdots,x_n)$,

ii) if $\al(x_0,\cdots,x_n)$ is defined then
$$\al(x_0,\cdots,\hat{x_i}\cdots,x_n):=\al(x_0,\cdots,x_{i-1},x_{i+1},\cdots,x_n)$$
is also defined and $\al(x_0,\cdots,x_n)\subset
\al(x_0,\cdots,\hat{x_i}\cdots,x_n)$.

iii) If $\al(x_0,\cdots,x_n)$ is defined and $x\in
\al(x_0,\cdots,x_n)$, then $\al(x_0,\cdots,x_n,x)$ is defined too.

iv) $\al(x)$ is defined for any point $x\in X$.

If $\al$ and $\bb$ are two Berishvili covers then we will say that
$\al$ is a \emph{refinement} of $\bb$ if every time when
$\al(x_0,\cdots,x_n)$ is defined then $\bb(x_0,\cdots,x_n)$ is also
defined and $\al(x_0,\cdots,x_n) \subset \bb(x_0,\cdots,x_n)$.

Having a Berishvili cover $\al$ one can form the  presimplicial
space $B\al$, which is given by $[n]\mapsto \bigsqcup
\al(x_0,\cdots, x_n)$ where the coproduct is taken over all
$(n+1)$-tuples of points $(x_0,\cdots, x_n)$ for which
$\al(x_0,\cdots, x_n)$ is defined. Now having any presheaf $P$ on
$X$ one defines the groups $H^*(\al,P)$ as the cohomology of the
cochain complex $C^*(\al,P)$ which is associated to the
precosimplicial abelian group $$[n]\mapsto
\prod_{(x_0,\cdots,x_n)}P(\al(x_0,\cdots,x_n))$$ where the product
is taken over all $(x_0,\cdots,x_n)\in X^{n+1}$ for which
$\al(x_0,\cdots,x_n)$ is defined. Observe that if $\al$ is a refinement
of $\bb$, then there is a canonical map of simplicial spaces
$B\al\to B\bb$, which induces the homomorphism $C^*(\bb,P)\to
C^*(\al,P)$ and by taking the colimit over all Berishvili covers of
$X$ one obtains the cochain complex $C^*(X,P)$, whose cohomology
groups are denoted by $H^*(X,P)$. Thus
$$H^*(X,P):={\sf colim}_\al H^*(\al,P).$$  Then we have the following fact.
\begin{The}\label{guram} \cite{beri} For any space $X$ and any sheaf $F$  there
is a canonical isomorphism
$$H^*(X,F)\cong \ext_{\sh(X)}^*(\Z,F).$$
\end{The}

\begin{proof} By the well-known axiomatic of derived functors
the result follows from ii) and v) of Lemma \ref{glema} below.
\end{proof}

Recall that a presheaf $P$ on $X$ is elementary if there exists a
collection of abelian groups $P_x$, $x\in X$ such that for any open
set $U\in OP(X)$ one has $$P(U) = \prod_{x\in U} P_x$$ with obvious
restriction morphisms. One easily observes that $P$ is in fact is a
sheaf.

\begin{Pro}\label{glema} \cite{beri} i) For any Berishvili cover $\al$ and for any
elementary presheaf $P$ one has
$$H^n(\al, P)=\begin{cases} P(X),& n=0\\ 0,&n\geq 1\end{cases}$$

ii)  If $I$ is an injective object in $\sh(X)$ then
$$H^n(X,I)=\begin{cases} I(X),& n=0\\ 0,&n\geq 1\end{cases}$$

iii) If $P$ is a presheaf such that $P^+=0$, then $H^*(X,P)=0$.

iv) If $P$ is a presheaf then $H^*(X,P)\cong H^*(X,P^+)$.

v) If $0\to F_1\to F\to F_2\to 0$ is a short exact sequence of
sheaves then one has a long exact  sequence of abelian groups:
$$0\to H^0(X,F_1)\to \cdots \to H^n(X,F)\to H^n(X,F_2)\to
H^{n+1}(X,F_1)\to H^{n+1}(X,F)\to\cdots $$

\end{Pro}

\begin{proof} i) Since the functors $H^n(\al, -):\psh(X)\to \ab$
respect products, it suffices to consider the case when the collection
of abelian groups $(P_x)_{x\in X}$ is nontrivial only in a given
point, say at $x_0\in X$. Then the ``evaluation at $x_0$'' gives
rise to the contraction of the sequence
$$0\to P_{x_0}\to C^0(\al, P)\to C^1(\al,P)\to \cdots$$

ii) For a sheaf $F$ we let $\tilde{F}$ be the elementary sheaf
generated by the collection of groups $F_x$. Then the canonical
morphism of sheaves $F\to \tilde{F}$ is a monomorphism. It follows
that if $I$ is an injective object in $\sh(X)$, then $I$ is a direct
summand of an elementary sheaf. Hence the result follows from i).

iii) In fact we will show that ${\sf colim}_\al C^*(\al,P)=0$.  Take
any $f \in C^p(\al,P)$ and any $(x_0,\cdots,x_p)\in X^{p+1}$ for
which $\al(x_0,\cdots,x_p)$ is defined. Since $P^+=0$ there exists an
open neighborhood $\beta(x_0,\cdots,x_p)$ of $x_p$ such that  the
image of $f(x_0,\cdots,x_p)\in P(\al(x_0,\cdots,x_p))$ in
$P(\beta(x_0,\cdots,x_p))$ is zero. Now we extend the function
$\beta$ by putting
$$\beta(x_0,\cdots,x_n)=\begin{cases} \al(x_0,\cdots,x_n) & {\rm if} \
n<p\\ \beta(x_0,\cdots,x_{n-1})\bigcap \al(x_0,\cdots,x_n) & {\rm
if} \ n>p\ {\rm and} \ x_n\in \beta(x_0,\cdots,x_{n-1})\end{cases}$$
One easily sees that $\beta$ is a Berishvili cover which is a
refinement of $\al$. By our construction image of $f$ in
$C^p(\beta,P)$ is zero and the result follows.

iv)  It is clear that the functors $H^n(\al, -):\psh(X)\to \ab$,
$n\geq 0$ form a $\delta$-sequence of functors. Thus the same is
true for $H^n(X,-):\psh(X)\to \ab$, $n\geq 0$. Observe that the
morphism $\xi:P\to P^+$ gives rise to the two   short exact
sequences of presheaves
$$0\to P_1\to P\to\Im(\xi)\to 0, \ \  \ \   \ \ 0\to \Im(\xi)\to P^+\to P_2\to 0$$
with $P_i^+=0$, $i=1,2$. The long cohomological exact sequences
together with iii) give the result.

iv) Since $F_2=P^+$, where $P$  fits in a short exact sequence of
presheaves
$$0\to F_1\to F\to P\to 0,$$
the result follows from iv).
\end{proof}

For paracompact $X$ similar facts are true also for \v{C}ech
cohomology. Hence  for such $X$ and an arbitrary $P\in \psh(X)$ one
has an isomorphism $\check{H}^i(X,P)\cong H^*(X,P)$, because both of
them are isomorphic to $Ext_{\sh(X)}(\Z,P^+)$. We will need more
direct construction of the isomorphism $\check{H}^i(X,P)\cong
H^*(X,P)$.
To do this and also for later use it is convenient
to use special covers and special maps of nerves
in the definition of
\v{C}ech cohomology.  Namely we consider  open covers of the form
$\u=(U_x)_{x\in X}$ where $x\in U_x$. If $\vu=(V_x)_{x\in X}$ is
another such cover, we will write  $\vu \leq \u$ if $V_x\subset U_x$ for
all $x\in X$. If this is the case, then we have a canonical map
$N\vu\to N\u$. Since special covers are cofinal in all covers, if we
take ${\sf colim}_\u H^*(\u,P)$ where $\u$ is running over special
covers and the canonical maps then we obtain again the \v{C}ech
cohomology $\check{H}^i(X,P)$. If $\u$ is such a cover of $X$, then
we can define a Berishvili cover $\al$ as follows. We put
$\al(x)=U_x$, and then by induction on $n$, we define
$\al(x_0,\cdots, x_n)=\al(x_0,\cdots,x_{n-1})\cap \al(x_n)$ provided
$x_n\in \al(x_0,\cdots,x_{n-1})$, otherwise $\al(x_0,\cdots, x_n)$
is not defined. One easily sees that in this way one gets
$H^*(\u,P)\cong H^*(\al,P)$. If one passes to the limit one obtains
the homomorphism $\check{H}^*(X,P)\to H^*(X,P)$. As we said this map
is an isomorphism if $X$ is a paracompact space.

\section{Cohomology with coefficients in prestacks}
\subsection{Preliminaries on prestacks  of abelian 2-groups}
We are assuming that the reader is familiar with the basic notions
and results on stacks and prestacks. Everything what we need  one
can find in \cite{ks}  or \cite{im}. Since the terminology in these
sources  diverges we recall the main definitions.

A \emph{prestack} $\p$ on  $X$ is  a contravariant pseudofunctor
from the category $OP(X)$ to the 2-category $\scg$. Thus it consists
of the following data:

i) for each open set $U$ an abelian 2-group $\p(U)$,

ii) for each pair of open sets $U\subset V$ a morphism of abelian
2-groups $r^V_U:\p(V)\to \p(U)$,

iii) for each triple of open subsets $U\subset V \subset W$ a track
in $\scg$
$$r^V_Ur^W_V\then r^W_U$$
satisfying the well-known properties (see Definition 19.1.3 in
\cite{ks}).

If $\p$ and $\q$ are prestacks on $X$ then  a \emph{morphism of
prestacks} $f:\p\to \q$ (\emph{functor of prestacks} in the
terminology \cite{ks}) is nothing but a pseudonatural
transformation, in other words it consists of:

1) for each open set $U$ a morphism of abelian 2-groups
$f(U):\p(U)\to \q(U)$

2) for each pair of open sets $U\subset V$ a track in $\scg$:
$$f(U)r^V_U \then r^V_U f(V)$$
satisfying the well-known properties (see Definition 19.1.4 in
\cite{ks})

If $f, g:\p\to \q$ are morphisms of prestacks, then a track
$\theta:f\then g$ (a \emph{morphism of functors of prestacks} in the
terminology of \cite{ks}) is nothing but a pseudomodification, in
other words it is given by tracks $\theta(U):f(U)\then g(U)$ in $\scg$ for each open set $U$, satisfying the well-known
condition (see Definition 19.1.5 in \cite{ks}).

Sometimes we will use the notation $a|_U$ instead of $r^V_U(a)$.

Observe that in many sources prestack is called ``fibred category''
(see for example \cite{im}) while the term ``prestack'' is used
for what we will call separated prestack.

Prestacks on $X$ together with the natural morphisms of prestacks
form an abelian 2-category $\pst(X)$ \cite{dupont}. Moreover
$\pst(X)$ possesses  enough projective and injective objects. This
easily follows from \cite{2der}. A morphism $\p\to \q$ of prestacks
is faithful (resp. cofaithful) if and only if for any open set $U$
the induced morphism $\p(U)\to \q(U)$ is  a faithful (resp.
cofaithful) morphism of abelian 2-groups. In particular a sequence
$$0\to \p_1\xto{i} \p\xto{p} \p_2\to 0$$
of prestacks  together with a track $\al:0\then pi$ is an extension
in $\pst(X)$ if and only if for any open subset $U$ of $X$ the
sequence
$$0\to \p_1(U)\to \p(U)\to \p_2(U)\to 0$$
together with the track $\al(U):0\then pi(U)$ is an extension  of
abelian 2-groups.

Let $\p$ be a prestack on $X$, then we obtain two presheaves
$\pi^0\p$ and $\pi^{-1}\p$ on $X$ by $U\mapsto \pi^i(\p(U))$,
$i=0,-1$.

\subsection{\v{C}ech and Berishvili cohomology with coefficients in
prestacks}

Let $\mathfrak{U}= \{U_i\}_{i\in I}$ be an open cover of an open set
$U$.  If $\p$ is a prestack on $X$ one obtains a precosimplicial
object in $\scg$:
$$[n]\mapsto \prod_{i_0,\cdots,i_n\in I}\p(U_{i_0\cdots i_n})$$
We let ${\bf H}^n(\u,\p )$ be the secondary cohomology of the
2-cochain complex $C^*(\u,\p)$  associated to it and we let
$H^n_U(\u,\p )$ be the corresponding Takeuchi-Ulbrich cohomology
groups. The obvious augmentation $N\u\to U$ yields the homomorphisms
$$\p(U)\to {\bf H}^0(\u,\p )$$
and passing to $\pi^i$, $i=0,-1$ one obtains the homomorphisms
$$\pi^0(\p(U))\to H^0_U(\u,\p ), \ \ \pi^{-1}(\p(U))\to H^{-1}_U(\u,\p
).$$

Moreover, we put
$$\check{{\bf H}}^n(X,\p):={\sf colim}_\u {\bf H}^n(\u,\p ),$$
where $\u$ varies over all special covers and canonical maps between
the corresponding nerves.  In this way one obtains pseudofunctors
$\check{{\bf H}}^n(X,-):\pst(X)\to \scg$, $n\in \Z$. Below we use
standard terminology of abelian 2-categories, see \cite{2-ch},
\cite{dupont}.

\begin{Pro}  i) If
$$0\to \p_1\to \p\to \p_2\to 0$$
is an extension of prestacks then
$$\cdots \to \check{{\bf H}}^n(X,\p_1 )\to \check{{\bf H}}^n(X,\p ) \to \check{{\bf
H}}^n(X,\p_2 ) \to \check{{\bf H}}^{n+1}(X,\p_1 )\to \cdots$$ is
part of a 2-exact sequence of abelian 2-groups, while
$$\cdots \to \check{H}^n_U(X,\p_1) \to \check{H}^n_U(X,\p ) \to \check{H}^n_U(X,\p_2) \to \check{H}^{n+1}_U(X,\p_1
)\to \cdots$$ is an exact sequence of abelian groups.

ii) For any prestack $\p$ there is an exact sequence of abelian
groups
$$\cdots \to \check{H}^{n+1}(X,\pi^{-1}(\p))\to \check{H}^n_U(X,\p)\to
\check{H}^n(X,\pi^0(\p))\to \check{H}^{n+2}(X,\pi^{-1}(\p))\to
\cdots
$$

\end{Pro}

\begin{proof} i) Let $\u$ be an open cover of $X$. Then one has
an extension of 2-chain complexes in $\scg$
$$0\to C^*(\u,\p_1)\to C^*(\u,\p)\to C^*(\u,\p_2)\to 0.$$
Thanks to \cite{2-ch} we obtain the following 2-exact sequence of
secondary cohomology:
$$\cdots \to {\bf H}^n(\u,\p_1 )\to {\bf H}^n(\u,\p ) \to {\bf H}^n(\u,\p_2
)\to {\bf H}^{n+1}(\u,\p_1 )\to \cdots $$ Since the filtered colimit
of 2-exact sequences of abelian 2-groups remains 2-exact  the result
follows. A similar argument based on the  TU-exact sequence for the
2-chain complex $C^*(\u,\p)$  gives ii).

\end{proof}

Observe that presheaves can be considered as discrete prestacks. So
we have an obvious inclusion $\psh(X)\subset \pst(X)$ and if we
restrict $\check{H}^n_U(X,-):\pst(X)\to \ab$ on $\psh(X)$ one
obtains the usual \v{C}ech cohomology.

If one takes Berishvili covers instead we obtain the well-defined
abelian 2-groups ${\bf H}^n(X,\p )$ and abelian groups $H^*_U(X,\p)$
with similar properties.

\begin{Pro}\label{3} i) If
$$0\to \p_1\to \p\to \p_2\to 0$$
is an extension of prestacks then
$$\cdots \to {\bf H}^n(X,\p_1 )\to {\bf H}^n(X,\p ) \to {\bf
H}^n(X,\p_2 ) \to {\bf H}^{n+1}(X,\p_1 )\to \cdots$$ is part of a
2-exact sequence of abelian 2-groups, while
$$\cdots \to H^n_U(X,\p_{1}) \to H^n_U(X,\p ) \to H^n_U(X,\p_2) \to H^{n+1}_U(X,\p_1
)\to \cdots$$ is an exact sequence of abelian groups.

ii) For any prestack $\p$ there is an exact sequence of abelian
groups
$$\cdots \to H^{n+1}(X,\pi^{-1}(\p))\to H^n_U(X,\p)\to
H^n(X,\pi^0(\p))\to H^{n+2}(X,\pi^{-1}(\p))\to \cdots
$$

iii) There is a  morphism of  abelian 2-groups $\check{{\bf
H}}^*(X,\p)\to {\bf H}^*(X,\p)$, which is an equivalence provided
$X$ is paracompact.

\end{Pro}

\begin{proof} i) and ii) have the same proofs as in the previous
case.  To prove iii) observe that, there is a morphism from \v{C}ech
cohomology to the Berishvili cohomology which is isomorphism for all
presheaves.  It follows from the 5-lemma and TU-exact sequence that
it induces isomorphism
$$\check{
H}_U^*(X,\p)\to H_U^*(X,\p)$$ for any $\p$ and the result follows.
\end{proof}

\subsection{Cohomology with coefficients in constant and elementary prestacks}
As we said any abelian 2-group
$\A$ gives rise to the constant prestack, denoted by $\A_c$.

\begin{Pro}\label{4} For a polyhedron $X$ one has an isomorphism
$$\check{H}^*_U(X,\A_c)\cong H^*(X,{\sf sp}(\A))$$
\end{Pro}
\begin{proof} We can assume that $X$ has a triangulation $T$. We take $\u$
to be the open cover of $X$ formed by the open stars of vertices of
$T$. It is well known that the nerve of this cover as a simplicial
set is isomorphic to a simplicial set $s(T)$ associated to $T$ (see
for example, Section 9.9 of \cite{es}). Thus we have isomorphisms
$$H^*(X,{\sf sp}(\A))=H^*(\A^{Sing_*(X)})\cong H^*(\A^{s(T)})\cong
\check{H}^*(\u,\A_c).$$ This gives a natural transformation
$H^*(X,{\sf sp}(\A)) \to \check{H}^*(X,\A_c)$. Since the
corresponding statement is well known for abelian groups, we can use
the TU-exact sequence and 5-lemma to finish the proof.
\end{proof}

A prestack  $\p$ on $X$ is \emph{elementary} if there exists a
collection of abelian 2-groups  $(\A_x)_{x\in X}$ such that for any
open set $X\in OP(X)$ we have
$$\p(U)=\prod_{x\in U}\A_x$$
with obvious restriction morphisms.

\begin{Le}\label{el} If $\p$ is an elementary prestack on $X$, then
$H^n_U(X,\p)=0$ if $n>0$. Moreover, we have ${\bf H}^n(X,\p)=0$ for
$n>1$ and ${\bf H}^1(X,\p)$ is a connected abelian 2-group.
\end{Le}

\begin{proof} If $\p$ is an elementary prestack, then $\pi^{i}\p$ is an
elementary presheaf, $i=0,-1$. Hence $H^n(X,\pi^i\p)=0$ for $n>0$
thanks to \ref{glema}. Now by the TU-exact sequence we get
$H^n_U(X,\p)=0$ if $n>0$. By definition $\pi^0({\bf
H}^n(X,\p))=H^n_U(X,\p)=0$ for $n>0$. On the other hand
$\pi^{-1}({\bf H}^n(X,\p))=H^{n-1}_U(X,\p)=0$ if $n>1$ and the
result follows.
\end{proof}

\section{Cohomology with coefficients in stacks}
\subsection{Preliminaries on stacks}

For a prestack $\p$ and a point $x\in X$ we put
$$\p_x={\sf colim}_{x\in U}\p(U).$$
Then $(-)_x:\pst\to \scg$ is an exact pseudofunctor for any $x\in
X$.

A morphism of prestacks $f:\p\to \q$ is called a \emph{weak
equivalence} (see Definition 2.3 of \cite{im}) if for any open
subset $U$ the induced functor $\p(U)\to \q(U)$ is fully faithful
and locally surjective on objects, in the sense that for any object
$a\in \q(U)$ and every $x\in U$ there exist an open set $V$ with
$x\in V\subset U$, an object $b\in \p(V)$ and an isomorphism
$f(V)(b)\to r^U_V(a)$.

For a prestack $\p$  we let $\Pi^i\p$ be the sheaves on $X$
associated to the presheaves $\pi^{i}\p$, $i=0,-1$
$$\Pi^i\p=(\pi^{i}\p)^+, \ \ i=0,-1.$$
We have $$(\Pi^i\p)_x\cong \pi^i(\p_x)$$ for all $x\in X$ and
$i=0,-1.$

\begin{Le}\label{if_w_e} If $f:\p\to \q$ is a weak
equivalence of prestacks then $\p_x\to \q_{x}$ is an equivalence
for all $x\in X$.

\end{Le}

\begin{proof} The induced map $\pi^i\p(U)\to \pi^i \q(U)$ is an
isomorphism for $i=-1$ and a monomorphism if $n=0$. Since
$$\pi^i(\p_x)={\sf colim}_{x\in U}\pi^i(\p(U)), \ \  \pi^i(\q_x)={\sf colim}_{x\in U}\pi^i(\q(U))$$
it follows that $\p_x\to \q_{x}$ yields an isomorphism
$\pi^{-1}(\p_x)\to \pi^{-1}(\q_x)$ and a monomorphism
$\pi^{0}(\p_x)\to \pi^{0}(\q_x)$. So far we used only full
faithfulness of the functors $\p(U)\to \q(U)$. Since $f$ is locally
surjective on objects the induced homomorphism  $\pi^{0}(\p_x)\to
\pi^{0}(\q_x)$  is obviously epimorphism and we are done.
\end{proof}

A prestack of abelian 2-groups $\p$ is called \emph{separated} (in
\cite{im} the corresponding objects are called simply prestacks) provided for any open
set $U$ and any objects $a,b$ of $\p(U)$ the presheaf $V\mapsto
Hom_{\p(V)}(a|_V,b|_V)$ is a sheaf on $U$. Here $V\subset U$ varies
over all open subsets. It follows that $\Pi^{-1}\p=\pi^{-1}\p$
provided $\p$ is separated.

A prestack of abelian 2-groups $\p$ is called a \emph{stack}
provided for any open set $U$ and any open cover $\mathfrak{U}=
\{U_i\}_{i\in I}$ the canonical morphism of abelian 2-groups
$\p(U)\to {\bf H}^0(\mathfrak{U},\p )$ is an equivalence of
categories. Observe that  the relative kernel of the diagram
obtained by taking the alternating sum of coface operators in
$$\prod_{i}\p(U_i)\begin{smallmatrix}\lra\\\lra\end{smallmatrix}\prod_{i,j}\p(U_{ij})\begin{smallmatrix}\lra\\\lra\\\lra\end{smallmatrix}\prod_{ijk}\p(U_{ijk})$$
is equivalent to the category od descent data \cite{im}. Hence our
definition is equivalent  to the classical definition of a stack
\cite{ks}. We let $\st$ be the full sub-2-category of $\pst$
consisting of stacks.

It is well known that the inclusion $\st\subset \pst$ has a left
adjoint $\p\mapsto \p^+$ which preserves relative kernels. Hence it
follows that  $\st$  is an abelian 2-category with enough injective
objects and the inclusion $\st\subset \pst$ respects (relative)
kernels \cite{ab-2-ab}.

\begin{Le}

i) If $\f$ is a stack then
$$\Pi^{-1}\f=\pi^{-1}\f .$$
ii) If additionally $\pi^{-1}\f=0$ then $$\Pi^0\f\cong \pi^0\f .$$

\end{Le}

\begin{proof} If $\f$ is a stack then it is  a separated prestack \cite{im}, hence
$\pi^{-1}\f$ is a sheaf and $\Pi^{-1}\f=\pi^{-1}\f$. The second part
 is obvious.
\end{proof}

\begin{Le}\label{zero} i)  For any prestack $\p$ the canonical map $\p\to \p^+$ yields
equivalences of abelian 2-groups
$$\p_x\to (\p^+)_x, \ \ x\in X.$$
In particular one has $\Pi^i(\p)=\Pi^i(\p^+)$.

 ii) For a prestack $\p$ the stack  $\p^+$ is equivalent to zero if and only if $\p_x$ is equivalent to zero for all $x\in X$.

 iii) The pseudofunctor $(-)_x:\st(X)\to\scg$ is exact.

iv) If $\p$ and $\q$ are stacks and $f:\p\to \q$ is a morphism of
stacks, then $f$ is an equivalence if and only if  $\p_x\to \q_{x}$ is
an equivalence for all $x\in X$.

v) A sequence of stacks $0\to \f_1\to \f\to \f_2\to 0$ is an
extension  in $\st(X)$ if and only if for any $x\in X$ the sequence
$0\to \f_{1x}\to \f_x\to \f_{2x}\to 0$ is an extension  of abelian
2-groups.
\end{Le}

\begin{proof} i) Recall that  $\p^+$  is constructed in two steps \cite{im}.
First  one constructs a separated prestack $\bar{\p}$. The objects
of $\bar{\p}(U)$ are the same as of $\p(U)$, while the morphism set
from $a\in\p(U)$ to $b\in \p(U)$ is the group of all sections on $U$
of the sheaf generated by the presheaf $V\mapsto
Hom_{\p(V)}(a|_V,b|_V)$ \cite{im}. It follows that
$$\pi^{-1}(\bar{\p}) =\Pi^{-1}\p$$
and hence $\pi^{-1}(\bar{\p}_x)=\pi^{-1}(\p_x)$, for all $x\in X$.
It is also clear from the description of $\bar{\p}$ that the map
$\pi^{0}(\p(U))\to \pi^0(\bar{\p}(U))$ is an epimorphism of abelian
groups and an object $a \in \p(U)$ lies in the kernel of this
homomorphism if and only if $a$ is locally isomorphic to $0$. It
follows that for any $x\in X$ the induced map $\pi^0(\p_x)\to
\pi^0(\bar{\p}_x)$ is an isomorphism. Hence $\p_x\to \bar{\p}_x$ is
an equivalence of abelian 2-groups. Next $\bar{\p}\to \p^+$ is a
weak equivalence (see Definition 2.7 \cite{im}). Thus $\bar{\p}_x\to
(\p^+)_x$ is an equivalence of abelian 2-groups and as a result $\p_x\to (\p^+)_x$ is too.

ii) It follows from the previous lemma that if $\p^+=0$ then
$\p_x=0$. Conversely, assume $\p_x=0$ for all $x\in X$. Then
$\Pi^i\p=0=\Pi^i(\p^+)$, $i=0,1$. So $\pi^{-1}(\p^+)=0$. Hence
$\Pi^0(\p^+)=\pi^0(\p^+)=0$ and the result follows.

iii) The fact that $(-)_x$ preserves kernels is obvious, because the
inclusion $\pst(X)\hookrightarrow \st(X)$ does preserve kernels and
$(-)_x:\pst(X)\to \scg$ is exact. Assume  $ \f_1\xto{i} \f\xto{p}
\f_2$, together with a track $0\then pi$  is the cokernel of $i$ in
the abelian 2-category $\st(X)$. Then $\f_2=\p^+$ where $p$ is the
cokernel of $i$ in $\pst(X)$. Then for an open set $U$ the abelian
2-group $\p(U)$ is the cokernel of $ \f_1(U)\to \f(U)$. Hence for
any $x\in X$ the abelian 2-group $\p_x$ is the cokernel of $
\f_{1x}\to \f_x$. Apply now the part i) to deduce the result.

iv) If $f:\p\to \q$ is an equivalence then it is of course also a
weak equivalence. Hence by Lemma \ref{if_w_e} $f_x$ is an
equivalence for all $x\in X$. Conversely, assume $f_x$ is an
equivalence for all $x\in X$. Let  $\p_1$ be the kernel of $f$ and
let $\q_1$ be the cokernel of $f$. Then by iii) $\p_{1x}=\q_{1x}=0$.
Hence $\p_1=\q_1=0$ by ii) and we are done.

v)  By iii) the ``if'' part is clear.  Assume $0\to \f_1\xto{i} \f\xto{p}
\f_2\to 0$, together with a track $0\then pi$  is a sequence of
stacks such that for all $x\in X$ the sequence $0\to \f_{1x}\to
\f_x\to \f_{2x}\to 0$ together with induced tracks $0\then p_xi_x$ is
an extension of abelian 2-groups. We claim that $\f_1\to \f$ is
faithful in $\st(X)$. To show this it is equivalent to show  that $\f_1(U)\to
\f(U)$ is faithful in $\scg$. Observe that we have a commutative
diagram
$$\xymatrix{\pi^{-1}(\f_1(U))\ar[r]\ar[d]&\pi^{-1}(\f(U))\ar[d]\\ \prod_{x\in
U}\pi^{-1}(\f_{1x})\ar[r]& \prod_{x\in U}\pi^{-1}(\f_{x})}$$ Since
$\pi^{-1}\f$ and $\pi^{-1}\f_1$ are sheaves  the vertical arrows are
monomorphisms. By assumption the bottom arrow is also a monomorphism
and the claim follows. For a moment let us denote $\f_3$ the
cokernel of $\f_1\to \f$. Then we have a morphism $\f_3\to \f_2$. By
iii) it induces an equivalence $\f_{3x}\to F_{2x}$ for all $x\in X$,
hence it is an equivalence by iv).

\end{proof}

\subsection{Cohomology with coefficients in stacks}
In this section we show that the prestack cohomology defined in the
previous section is in fact determined by associated stacks.

\begin{Pro}\label{stCoh} i)
If $\p$ is a prestack, then $${\bf H}^*(X,\p)\cong {\bf
H}^*(X,\p^+)$$

 ii) If
$$0\to \f_1\to \f\to \f_2\to 0$$
is an extension of stacks then
$$\cdots \to {\bf H}^n(X,\f_1 )\to {\bf H}^n(X,\f ) \to {\bf
H}^n(X,\f_2 ) \to {\bf H}^{n+1}(X,\f_1 )\to \cdots$$ is part of a
2-exact sequence of abelian  2-groups, while
$$\cdots \to H^n_U(X,\f_1) \to H^n_U(X,\f ) \to H^n_U(X,\f_2) \to H^{n+1}_U(X,\f_1
)\to \cdots$$ is an exact sequence of abelian groups.

iii) For any stack $\f$ there is an exact sequence of abelian groups
$$\cdots \to H^{n+1}(X,\Pi^{-1}(\f))\to H^n_U(X,\f)\to
H^n(X,\Pi^0(\f))\to H^{n+2}(X,\Pi^{-1}(\f))\to \cdots
$$

iv) For paracompact $X$ one has similar results for the groups
$\check{{\bf H}}^*(X,\f)$.

\end{Pro}

\begin{proof} i) By Lemma \ref{TUzero} it suffices to show that
$H^n_U(X,\p)\to H^n_U(X,\p^+)$ is an isomorphism for all $n$. To see
this we use part 2 of Proposition \ref{3}. Thus we have the
commutative diagram of abelian groups with exact rows
$$\xymatrix@C=1em{
\cdots \ar[r]& H^{n+1}(X,\pi^{-1}(\p))\ar[r]\ar[d]&
H^n_U(X,\p)\ar[r]\ar[d]& H^n(X,\pi^0(\p))\ar[r]\ar[d]&
H^{n+2}(X,\pi^{-1}(\p))\ar[r]\ar[d] &\cdots\\ \cdots \ar[r]&
H^{n+1}(X,\pi^{-1}(\p^+))\ar[r]& H^n_U(X,\p^+)\ar[r]&
H^n(X,\pi^0(\p^+))\ar[r]& H^{n+2}(X,\pi^{-1}(\p^+))\ar[r] &\cdots}
$$
By the 5-lemma it suffices to prove that $H^*(X,\pi^i(\p))\to
H^*(X,\pi^i(\p^+))$ is an isomorphism. But this follows from
part iv) of Proposition \ref{glema} together with isomorphisms
$(\pi^i(\p))_x\cong \pi^i(\p_x)\cong \pi^i(\p^+_x)\cong
(\pi^i(\p^+))_x$.

ii) By definition $\f_2=\p^+$, where $\p$ is the cokernel of
$\f_1\to \f$ in $\pst(X)$. By Proposition \ref{3} one has 2-exact
sequences involving the 2-groups ${\bf H}^n(X,\f_1 )$, $ {\bf
H}^n(X,\f )$ and $ {\bf H}^n(X,\p )$. Hence the result follows from
part i). Similar arguments work for iii) and iv).
\end{proof}

\subsection{Secondary Ext}
We start with the following familiar construction.

\begin{Le}\label{ist} i) Fix a point $x\in X$ and an abelian 2-group
$\A$. Define $i_x(\A)$ to be the stack given by
$$i_x(\A)(U)=\begin{cases}\A &  {\rm if} \ \  x\in U\\ 0, &  {\rm if} \ \ x \not \in
U\end{cases}$$ Then for any stack $\f$ one has
$${\bf Hom}_{\st(X)}(\f, i_x(\A))\simeq {\bf Hom}_{\scg}(\f_x,\A)$$

ii) If $(\A_x)_{x\in X}$ is a collection of injective abelian
2-groups, then $\prod_{x\in X}i_x(\A_x)$ is an injective object in
$\st(X)$.

iii) If $\f$ is a stack then there exists an extension
$$0\to \f\to \f_1\to \f_2\to 0$$
with injective $\f_1$.

iv) Any injective stack is a direct summand of an elementary stack.
\end{Le}

\begin{proof} i) is the matter of a straightforward checking, ii) is
a direct consequence of i) and Lemma \ref{zero}. To show iii) we
choose extensions of abelian 2-groups $$0\to \f_x\to\A_x\to\B_x\to
0$$ with $\A_x$ an injective abelian 2-group. Then by i) we have a
morphism $\f\to \prod_{x\in X}i_x(\A_x)$ and we can take
$\f_1=\prod_{x\in X}i_x(\A_x)$. Finally iv) follows from the above and
\cite{ext-b-v}, Corollary 11.2.
\end{proof}

Let $\Phi$ be the Picard category constructed in \cite{2der} which we described in the introduction. For an
abelian  2-group $\A$, the stack $(\A_c)^+$ is called the constant
stack corresponding to $\A$. By abuse of notation we write $\A$
instead of $(\A_c)^+$.

Now we are in the position to prove our main theorem which relates
the secondary $Ext$ from  \ref{Sec_ext} to our cohomology theory.
Since $\st(X)$ has enough injective objects we can apply the
construction from  \ref{Sec_ext}  to get for any stacks $\f_1,\f_2$
the abelian 2-group ${\bf Ext}^*_{\st(X)}(\f_1,\f_2)$. As usual with
Takeuchi-Ulbrich objects we have
$$ Ext^*_{\st(X)}(\f_1,\f_2):=\pi^0({\bf
Ext}^*_{\st(X)}(\f_1,\f_2)).$$

\begin{The}\label{ext} For any stack $\f$ one has a natural equivalence
$${\bf Ext}^*_{\st (X)}(\Phi,\f)\cong {\bf H}^*(X,\f).$$
\end{The}

\begin{proof} Observe that
$${\bf Hom}_{\st}(\Phi,\f)\cong {\bf Hom}_{\pst}((\Phi)_c,\f)\cong {\bf Hom}_{\scg}(\Phi,\f(X))\cong \f(X)\cong {{\bf H}}^0(X,\f).$$
By the axiomatic characterization of the secondary derived functors
(see Proposition  \ref{ax}) it suffices to show that if $\f$ is an
injective object then ${{\bf H}}^n(X,\f)=0$ for $n>1$ and ${\bf
H}^1(X,\f)$ is connected. By Lemma \ref{ist} this would follow if we
prove a similar statement for elementary stacks. By ii) of
Proposition \ref{stCoh} this reduces to the case of elementary
prestacks which was handled in Lemma \ref{el}.
\end{proof}

Now we have all the ingredients to prove the following result.

\begin{The} Let $\A$ be an abelian 2-group. Then for any polyhedron
$X$ one has an isomorphism
$$H^*(X,{\sf sp}(\A))\cong Ext^*_{\st(X)}(\Phi,\A).$$
\end{The}

\begin{proof} By Theorem \ref{ext} and ii) Proposition \ref{stCoh} we
have $$Ext^*_{\st(X)}(\Phi,\A) \cong {H}^*_U(X,\A)\cong
H^*_U(X,\A_c).$$ On the other hand $H^*(X,{\sf sp}(\A))\cong
\check{H}^*_U(X,\A_c)$ thanks to Proposition \ref{4}. Hence the
result follows from iii) Proposition \ref{3}.
\end{proof}

This result is  a companion of the following result which is a
trivial consequence of Theorem 16  \cite{2der}. Assume $X$ is a
topological space such that $\pi_i(X)=0$ for $i\not = 1,2$. It is
well known that the fundamental groupoid of the loop space of $X$
has a 2-group structure. Denote this 2-group by $\G$ and again let
$\A$ be the abelian 2-group corresponding to a 2-stage spectrum $E$.
Then $$H^*(X,E)\cong Ext^*(\Phi, \A)$$ where $Ext$ is now taken in
the 2-category of 2-representations of $\G$ and the actions of $\G$
on $\Phi$ and $\A$ are trivial. Compare this  result with the
familiar fact that the cohomology of a $K(\Pi,1)$-space can be
computed as $Ext$ in the category of representations of the group
$\Pi$.
\subsection{Line bundles, discriminant and twisted Sheaves} One of the most important stacks of
abelian 2-groups is given by line bundles on a manifold $X$. In this
section we consider the cohomology with coefficients in this
particular stack. As we will see soon this cohomology up to shift in
the dimension is the same as the cohomology with coefficients in the
sheaf of invertible elements, so we do not get anything new. But
existence of this isomorphism depends on the fact that any
invertible module over a local ring is trivial. Hence we get more
interesting situation when we consider more general objects than
manifolds. Let $(X,{\mathcal O}_X)$ be a ringed space. Thus
${\mathcal O_X}$ is a sheaf of commutative rings on $X$. Then we
have a stack ${\mathcal L}$ of invertible ${\mathcal O}_X$-modules.
This stack assigns to an open set $U$ of $X$ the groupoid of
invertible ${\mathcal O}_X(U)$-modules and their isomorphisms. The
tensor product equips this stack with a structure of a stack of
abelian 2-groups. So we have well-defined groups $H_U^*(X,{\mathcal
L}_X)$. It is clear that we have a canonical isomorphism of sheaves
$$\Pi^{-1}({\mathcal L}_X)\cong {\mathcal O^*}_X$$
and for any $x\in X$ one has an isomorphism of abelian groups
$$(\Pi^0({\mathcal L}_X))_x\cong {\sf Pic}({\mathcal O}_x).$$
The TU-exact sequence gives us the following exact sequences of
abelian groups
$$\cdots \to H^{n+1}(X,{\mathcal O^*}_X)\to H^n_U(X, {\mathcal L}_X)\to
H^n(X,\Pi^0({\mathcal L}_X))\to H^{n+2}(X,{\mathcal O^*}_X)\to
\cdots$$ In particular we have
$$ H^n_U(X, {\mathcal L}_X)\cong H^{n+1}(X,{\mathcal O^*}_X)$$
provided ${\mathcal O}_x$ is a local ring for all $x$. This is so
for example, when $X$ is a scheme or a complex manifold. However for
general $(X,{\mathcal O}_X)$ these groups are different. It is well
known that the groups $H^*(X,{\mathcal O^*}_X)$ appear in many
problems of geometry and hopefully the same is true for $H^n_U(X,
{\mathcal L}_X)$. In this direction let us mention the following
fact which is a restatement of a result of Section 19.6 \cite{ks}.
\begin{Pro} The set of equivalence classes of stacks of twisted
${\mathcal O}_X$-modules is isomorphic to $H^1_U(X, {\mathcal
L}_X)$.
\end{Pro}
In fact this was proved implicitly in \cite{ks} (see Remark 19.6.4
(iii) in \cite{ks}) modulo the fact that instead of group $H^1_U(X,
{\mathcal L}_X)$ they have $H^2(X,{\mathcal O}^*_X)$. However they
are assuming that ${\sf Pic}({\mathcal O}_x)=0$ for all $x$. But
this last restriction they have only at the very end and the
argument before that proves precisely the statement we just state.

Recall also that a \emph{discriminant} \cite{knus} over
$(X,{\mathcal O}_X)$ is a pair $(L,h)$, where $L$ is an invertible
${\mathcal O}_X$-module and $h$ is a nondegenerate symmetric
bilinear  form $h:L\otimes L\to {\mathcal O}_X$. Here the tensor
product is taken over ${\mathcal O}_X$.  A morphism from $(L,h)$ to
$(L',h')$ is an isomorphism of ${\mathcal O}_X$-modules which is
compatible with forms. The isomorphism classes of discriminants is
denoted by ${\sf Dis}({\mathcal O}_X)$. In this way one obtains the
stack of discriminants ${\mathcal D}_X$. Since $h:L\otimes L\to
{\mathcal O}_X$ is an isomorphism, we see that the stack ${\mathcal
D}_X$ is nothing but the kernel of ${\mathcal L}_X\xto{2} {\mathcal
L}_X$. Here $2$ is used in additive notations, in multiplicative
notation it is given by $L\mapsto L\otimes L$. Assume for any $x\in
X$ the group ${\sf Pic}({\mathcal O}_x)$ is 2-divisible, meaning
that any element is of the form $x^2$.  For instance this obviously
holds if ${\mathcal O}_x$ is a local ring. Then we get an extension
of stacks:
$$0\to {\mathcal D}_X\to{\mathcal L}_X\xto{2} {\mathcal
L}_X\to 0$$ which yields not only  a well-known short exact sequence
\cite{knus}
$$0\to {\mathcal O}^*/{\mathcal O}^{*2}\to {\sf Dis}(X)\to \, _2{\sf Pic}(X)\to 0$$
but also gives a  function which assigns to  each invertible
${\mathcal O}_X$-module $L$ (resp. stack $T$ of twisted ${\mathcal
O}_X$-modules) a class $d(L)\in H^1_U(X,{\mathcal D}_X)$ (resp.
$d(T)\in H^2_U(X,{\mathcal D}_X$)) which vanishes if and only if
there exists a ''square root'' of $L$ (resp. $T$). Of course $2$ can
be replaced by any integer $n$. The role of the stack ${\mathcal
 D}$ will be play by the kernel of ${\mathcal L}_X\xto{n} {\mathcal L}_X$.

\end{document}